\documentclass[a4paper]{amsart}
\usepackage{a4wide}

\usepackage{graphicx, amsthm, amsmath, amssymb, mathrsfs, amscd}
\usepackage{amsmath}
\usepackage{breakurl}
\usepackage[tableposition=below]{caption}
\usepackage{enumitem}
\usepackage{hyperref}
\usepackage{color}
\usepackage{color, colortbl}
\usepackage{tikz-cd}
\usepackage{stmaryrd}

\newtheorem{theorem}{Theorem}[section]

\newtheorem{cor}[theorem]{Corollary}

\theoremstyle{definition}

\theoremstyle{definition}

\title{Algebraic Geometry over Non-Algebraically Closed Fields:\\
A-Coherent Sheaves over a Ringed Space}

\author[Hamet Seydi]{Hamet Seydi}
\address{Emeritus Professor, Department of Mathematics\\
Cheikh Anta Diop University\\
Dakar\\
Senegal.}
\email{hsedi@gmail.com}

\author[Teylama Miabey]{Teylama Miabey}
\address{Department of Mathematics\\
University of the District of Columbia\\
Washington DC\\
United States.}
\email{teylama.miabey@udc.edu}

\keywords{Algebraic geometry; coherent sheaves; flatness; Nash functions; category equivalence}
\subjclass[2020]{14F05, 14A10, 13B40, 32C07}

\begin{document}

\begin{abstract}
In this paper, we investigate the properties of $A$-coherent and $A$-quasi-coherent sheaves within the framework of algebraic geometry over non-algebraically closed fields. We define an $\mathcal{O}_X$-module to be $A$-coherent (resp. $A$-quasi-coherent) if it admits a global presentation by free modules of finite rank (resp. arbitrary rank) over a ringed space $X$. We establish a fundamental correspondence between these sheaves and modules over the ring of global sections $A = \Gamma(X,\mathcal{O}_X)$. Specifically, we prove that under conditions of flatness for the canonical morphism and the exactness of the global section functor, there exists an equivalence of categories between $A$-coherent $\mathcal{O}_X$-modules and finitely presented modules over $A$. We further demonstrate the utility of these results by proving the faithful flatness of the canonical homomorphisms from rings of Nash functions to rings of analytic functions, utilizing the vanishing of higher cohomology groups as guaranteed by Cartan's Theorem~B.
\end{abstract}
\medskip

\maketitle

\section{Introduction and Background}

The study of sheaves on ringed spaces traditionally relies on the assumption of an algebraically closed base field. The foundational work of Serre~\cite{SerreFAC} in \emph{Faisceaux Alg\'ebriques Coh\'erents} (FAC) established the equivalence between the category of quasi-coherent sheaves and modules over the coordinate ring for affine schemes. However, extending this correspondence to real algebraic geometry and non-algebraically closed fields introduces significant complexities.

A primary tool in bridging local and global data is Cartan's Theorem~B~\cite{Cartan}, which guarantees the vanishing of higher sheaf cohomology $(H^i = 0$ for $i>0)$ for coherent analytic sheaves on Stein manifolds. In the real case, the development of Nash functions---functions that are both algebraic and analytic---provides a setting where ``Stein-like'' properties often hold. The relationship between the ring of Nash functions and the ring of analytic functions was significantly advanced by Raimondo~\cite{Raimondo}, who established the faithful flatness of the inclusion.

In this paper, we define $A$-coherence, a global property that ensures a sheaf is the pullback of a finitely presented module from the spectrum of the global section ring. By leveraging the cohomological vanishing provided by the Cartan--Serre framework, we establish a robust categorical equivalence for $A$-coherent sheaves.

\section{Main Results}

Let $(X, \mathcal{O}_X)$ be a ringed space and $F$ an $\mathcal{O}_X$-module.

\begin{enumerate}
    \item We say that $F$ is \emph{$A$\nobreakdash-coherent} if there exists an exact sequence
    \[
    \mathcal{O}_X^{m} \longrightarrow \mathcal{O}_X^{n} \longrightarrow F \longrightarrow 0
    \]
    where $m$ and $n$ are integers $\geq 1$.
    
    \item We say that $F$ is \emph{$A$\nobreakdash-quasi-coherent} if there exists an exact sequence
    \[
    \mathcal{O}_X^{I_1} \longrightarrow \mathcal{O}_X^{I_2} \longrightarrow F \longrightarrow 0
    \]
    where $I_1$ and $I_2$ are sets.
\end{enumerate}

We start with the following result:

\begin{theorem}[I.1]
Let $(X, \mathcal{O}_X)$ be a locally ringed space, $Y = \mathrm{Spec}(\Gamma(X, \mathcal{O}_X))$, and $F$ an $\mathcal{O}_X$-module that is $A$\nobreakdash-coherent (resp. $A$\nobreakdash-quasi-coherent). Then there exists an $\mathcal{O}_Y$-module of finite presentation (resp. quasi-coherent) $G$ such that $F \cong \varphi^{\ast} G$, where $\varphi : X \to Y$ is the canonical morphism.

Conversely, if $G$ is an $\mathcal{O}_Y$-module of finite presentation (resp. quasi-coherent), then $F = \varphi^{\ast} G$ is an $\mathcal{O}_X$-module that is $A$\nobreakdash-coherent (resp. $A$\nobreakdash-quasi-coherent).
\end{theorem}

\begin{proof}
By hypothesis, we have an exact sequence
\[
\mathcal{O}_X^{I_1} \xrightarrow{u} \mathcal{O}_X^{I_2} \longrightarrow F \longrightarrow 0
\]
where $I_1$ and $I_2$ are finite if $F$ is $A$\nobreakdash-coherent.  

Set
\[
M = \mathrm{coker}\big( \Gamma(X, \mathcal{O}_X)^{I_1} \xrightarrow{\Gamma(u)} \Gamma(X, \mathcal{O}_X)^{I_2} \big)
\]
and let $G = \widetilde{M}$ be the $\mathcal{O}_Y$-module associated to $M$.

Then we have an exact sequence
\[
\mathcal{O}_Y^{I_1} \xrightarrow{v} \mathcal{O}_Y^{I_2} \longrightarrow G \longrightarrow 0
\]
where $v = \Gamma(X, u)$. This induces the exact sequence
\[
\mathcal{O}_X^{I_1} \xrightarrow{\varphi^{\ast} v} \mathcal{O}_X^{I_2} \longrightarrow \varphi^{\ast} G \longrightarrow 0.
\]
But $\varphi^{\ast} v$ identifies with $u$, hence the canonical morphism $\varphi^{\ast} G \to F$ is an isomorphism.
\end{proof}


\noindent
Conversely, if $G$ is a finitely presented (resp. quasi–coherent) $\mathcal{O}_Y$–module, it is clear that $F = \varphi^* G$ is an $\mathcal{O}_X$–module $A$–coherent (resp. $A$–quasi–coherent).
\qed

\begin{theorem}[I.2]
Under the same notations and assumptions as in (I.1), suppose the following conditions hold:
\begin{enumerate}
\item The ring $\Gamma(X, \mathcal{O}_X)$ is Noetherian (resp. $\mathcal{O}_X$ is coherent).
\item $\varphi$ is a flat morphism.
\item The functor $F \mapsto F(X, F)$ from the category of $\mathcal{O}_X$–modules $A$–coherent (resp. $\mathcal{O}_X$–modules coherent) to the category of $\Gamma(X, \mathcal{O}_X)$–modules is exact.
\end{enumerate}
Then the functor $G \mapsto \varphi^* G$ from the category of $\mathcal{O}_Y$–modules coherent (resp. $\mathcal{O}_Y$–modules finitely presented) to the category of $\mathcal{O}_X$–modules $A$–coherent is an equivalence of categories.
\end{theorem}

\begin{proof}
Let $F$ be an $\mathcal{O}_X$–module $A$–coherent. Then there exists an exact sequence
\[
\mathcal{O}_X^m \xrightarrow{u} \mathcal{O}_X^n \longrightarrow F \longrightarrow 0
\]
where $u = \varphi^* v$ for some morphism $v$ of $\mathcal{O}_Y$–modules $\mathcal{O}_Y^m \to \mathcal{O}_Y^n$.  

We deduce that
\[
\ker(u) = \varphi^*(\ker(v)), \quad \mathrm{Im}(u) = \varphi^*(\mathrm{Im}(v)),
\]
since $\varphi$ is flat. As $\ker(u)$ and $\mathrm{Im}(u)$ are coherent if $\mathcal{O}_X$ is coherent, we deduce in both cases that the sequences
\[
0 \to \Gamma(X, \ker(u)) \to \Gamma(X, \mathcal{O}_X^m) \to \Gamma(X, \mathrm{Im}(u)) \to 0
\]
and
\[
0 \to \Gamma(X, \mathrm{Im}(u)) \to \Gamma(X, \mathcal{O}_X^n) \to \Gamma(X, F)
\]
are exact. We thus conclude that the sequence
\[
\Gamma(X, \mathcal{O}_X^m) \to \Gamma(X, \mathcal{O}_X^n) \to \Gamma(X, F)
\]
is exact.

From the proof of (I.1), if $G = F(X, F)^\sim$, then $F \cong \varphi^* G$. Moreover, if $F_1 \to F_2$ is a morphism of $\mathcal{O}_X$–modules $A$–coherent, it is clear that $u = \varphi^* f$ where $\Gamma(X, u)^\sim$, hence the conclusion.
\end{proof}

\begin{cor}[I.2.1]
Under the same notations and assumptions as in Theorem (I.2), the functor
\[
F \longmapsto \varphi^* F
\]
from the category of quasi–coherent $\mathcal{O}_Y$–modules to the category of $A$–quasi–coherent $\mathcal{O}_X$–modules is an equivalence of categories.
\end{cor}

\begin{proof}
Indeed, from (I.2), any $A$–quasi–coherent $\mathcal{O}_X$–module $F'$ is isomorphic to $\varphi^* F$ where $F = \Gamma(X, F')^\sim$ is a quasi–coherent $\mathcal{O}_Y$–module.  

It is clear that if $G'$ is another $A$–quasi–coherent $\mathcal{O}_X$–module and $G = \Gamma(X, G')^\sim$, then the map
\[
f \mapsto \varphi^* f
\]
from $\mathrm{Hom}_{\mathcal{O}_Y}(F, G)$ to $\mathrm{Hom}_{\mathcal{O}_X}(F', G')$ is bijective, hence the conclusion.
\end{proof}

\begin{cor}[I.2.2]
With the same notations as (I.2) and hypotheses 1) and 2) of (I.2), the following are equivalent:
\begin{enumerate}
\item[(3)] $\varphi(X)$ contains $\mathrm{Max}(Y)$.
\item[(i)] The functor $F' \mapsto \Gamma(X, F')$ from the category of $\mathcal{O}_X$–modules $A$–coherent to the category of $\Gamma(X, \mathcal{O}_X)$–modules is exact.
\item[(ii)] The map $\Gamma(Y, F) \to \Gamma(X, \varphi^* F)$ is an isomorphism for every finitely presented $\mathcal{O}_Y$–module $F$.
\end{enumerate}
\end{cor}

\begin{proof}
(1) $\Rightarrow$ (ii): Let $F$ be a finitely presented $\mathcal{O}_Y$–module. Then there exists an exact sequence
\[
\mathcal{O}_Y^m \longrightarrow \mathcal{O}_Y^n \longrightarrow F \longrightarrow 0
\]
which yields the exact sequence
\[
\mathcal{O}_X^m \longrightarrow \mathcal{O}_X^n \longrightarrow \varphi^* F \longrightarrow 0.
\]
Applying $\Gamma(Y, -)$ and $\Gamma(X, -)$ respectively to these sequences, we deduce that the map $\Gamma(Y, F) \to \Gamma(X, \varphi^* F)$ is bijective since $\Gamma(Y, \mathcal{O}_Y) = \Gamma(X, \mathcal{O}_X)$.

(ii) $\Rightarrow$ (i): Let $F'$ be an $A$–coherent $\mathcal{O}_X$–module. Then $F' \cong \varphi^* F$ where $F$ is a finitely presented $\mathcal{O}_Y$–module. Thus
\[
\Gamma(X, F') = \Gamma(X, \varphi^* F) \cong \Gamma(Y, F),
\]
so $F' = \Gamma(X, F')^\sim$. Therefore, if $0 \to F_1' \to F_2' \to F_3' \to 0$ is an exact sequence of $A$–coherent $\mathcal{O}_X$–modules, the sequence
\[
0 \to \Gamma(X, F_1') \to \Gamma(X, F_2') \to \Gamma(X, F_3')
\]
is exact, hence $\Gamma(X, -)$ is exact on $A$–coherent $\mathcal{O}_X$–modles.
\end{proof}

\begin{cor}[I.2.3]
Under the same notations and assumptions as in (I.2), let $F$ be an $\mathcal{O}_X$–module.  
The following are equivalent:
\begin{enumerate}
\item $F$ is an $A$–quasi–coherent $\mathcal{O}_X$–module.
\item $F$ is a direct limit of $A$–coherent $\mathcal{O}_X$–modules.
\end{enumerate}
\end{cor}

\begin{proof}
$(1) \Rightarrow (2)$ follows from Theorem (I.1) since every finitely presented quasi–coherent $\mathcal{O}_Y$–module is a direct limit of $A$–coherent $\mathcal{O}_Y$–modules.  

Conversely, if $F$ is an $A$–coherent $\mathcal{O}_X$–module, then the functor $\Gamma(X, -)$ is exact on the category of $A$–coherent $\mathcal{O}_X$–modules (by hypothesis). Thus, writing $F' = \varphi^* F^\prime$ where $F^\prime = \Gamma(X, F)^\sim$, we have
\[
F' = \varinjlim F_\alpha \quad \text{with} \quad F = \varinjlim F_\alpha,
\]
and the conclusion follows.
\end{proof}

\begin{theorem}[I.3]
Let $A$ be a ring (resp. a Noetherian ring), $Y = \mathrm{Spec}(A)$, and $f \colon X \to Y$ a morphism of ringed spaces. Suppose:
\begin{enumerate}
\item $f$ is flat.
\item For every quasi–coherent (resp. coherent) $\mathcal{O}_Y$–module $F$, we have $H^1(X, f^* F) = 0$.
\end{enumerate}
Then the canonical homomorphism
\[
A = \Gamma(Y, \mathcal{O}_Y) \longrightarrow B = \Gamma(X, \mathcal{O}_X)
\]
is flat.
\end{theorem}

\begin{proof}
Let $I$ be an ideal of $A$. Then there is an exact sequence
\[
A^{I_1} \to A^{I_2} \to I \to 0
\]
(where we may take $I_1$ and $I_2$ finite when $A$ is Noetherian), which yields the exact sequence
\[
\mathcal{O}_Y^{I_1} \xrightarrow{u} \mathcal{O}_Y^{I_2} \xrightarrow{v} \tilde{I} \to 0.
\]
We deduce the exact sequence
\[
\mathcal{O}_X^{I_1} \xrightarrow{u'} \mathcal{O}_X^{I_2} \xrightarrow{v'} f^* \tilde{I} \to 0.
\]
\end{proof}

\begin{cor}[I.2.3]
With the notations and hypotheses of (I.2), let $F$ be an $\mathcal O_X$–module.
The following are equivalent:
\begin{enumerate}
\item $F$ is an $\mathcal O_X$–module that is $A$–quasi–coherent;
\item $F$ is a filtered inductive limit of $\mathcal O_X$–modules that are $A$–coherent.
\end{enumerate}
\end{cor}

\begin{proof}
(1)$\Rightarrow$(2) follows from Theorem (I.1), since every $\mathcal O_Y$–quasi–coherent module of finite presentation is a filtered inductive limit of $\mathcal O_Y$–modules of finite presentation.  
Conversely, if $F$ is an $\mathcal O_X$–module that is $A$–coherent, then the functor
\[
\Gamma(X, -)
\]
is exact in the category of $\mathcal O_X$–modules that are $A$–coherent (by hypothesis). Hence for $F' = \varphi^*F$ we have
\[
F' = \varinjlim F_\alpha \qquad\text{and}\qquad \Gamma(X, F')=\varinjlim \Gamma(X,F_\alpha),
\]
so the system $(F_\alpha)$ is filtered and $F \simeq \varinjlim F_\alpha$, which proves the claim.
\end{proof}

\begin{theorem}[I.3]
Let $A$ be a Noetherian ring, let $Y=\operatorname{Spec}(A)$, and let $X$ be a ringed space. Let
\[
f:X \to Y
\]
be a morphism of ringed spaces. Assume that:
\begin{enumerate}
\item $f$ is flat;
\item for every quasi-coherent (resp.\ coherent) $\mathcal O_Y$-module $F$, one has
\[
H^1(X,f^*F)=0.
\]
\end{enumerate}
Then the canonical morphism
\[
A=\Gamma(Y,\mathcal O_Y)\longrightarrow B=\Gamma(X,\mathcal O_X)
\]
is flat.
\end{theorem}

\begin{proof}
Let $I$ be an ideal of $A$. There is an exact sequence $A^{I_1}\xrightarrow{u}A^{I_2}\xrightarrow{v} I\to 0$ (with $I_1,I_2$ finite when $A$ is Noetherian), which yields an exact sequence
\[
\mathcal O_Y^{I_1}\xrightarrow{u'}\mathcal O_Y^{I_2}\xrightarrow{v'} \widetilde I\to 0.
\]
Pulling back by $f$ gives an exact sequence
\[
\mathcal O_X^{I_1}\xrightarrow{u''}\mathcal O_X^{I_2}\xrightarrow{v''} f^*\widetilde I\to 0.
\]
Since $f$ is flat, $\ker(u'')=f^*(\ker(u'))$ and $\ker(v'')=f^*(\ker(v'))$.  
Hypothesis (2) implies $H^1(X,\ker(u''))=H^1(X,\ker(v''))=0$. Hence the sequence
\[
B^{I_1}\longrightarrow B^{I_2}\longrightarrow \Gamma(X,f^*\widetilde I)\longrightarrow 0
\]
is exact. But also
\[
B^{I_1}\longrightarrow B^{I_2}\longrightarrow I\otimes_A B\longrightarrow 0
\]
is exact, so the canonical map $I\otimes_A B\to \Gamma(X,f^*\widetilde I)$ is bijective. As $f^*\mathcal O_Y=\mathcal O_X$, the map $\mathcal O_Y\to f_*\mathcal O_X$ is injective, hence the induced map
\[
\Gamma(X,f^*\mathcal O_Y)=\Gamma(X,\mathcal O_X)=B
\]
is injective. Therefore $I\otimes_A B\to B$ is injective for every ideal $I$, which shows $A\to B$ is flat.
\end{proof}

But since $f$ is flat, we deduce the relations
\[
\ker(u'') = f^*\big(\ker(u')\big), \quad \ker(v'') = f^*\big(\ker(v')\big).
\]
Hypothesis (2) then implies
\[
H^1\big(X, \ker(u'')\big) = H^1\big(X, \ker(v'')\big) = 0.
\]
It follows that the sequence
\[
B^{I_1} \longrightarrow B^{I_2} \longrightarrow \Gamma\big(X, f^*\widetilde{J}\big) \longrightarrow 0
\]
is exact. Thus we have the exact sequence
\[
B^{I_1} \longrightarrow B^{I_2} \longrightarrow I \otimes_A B \longrightarrow 0.
\]
Hence the canonical map
\[
I \otimes_A B \longrightarrow \Gamma(X, f^*\widetilde{J})
\]
is bijective. As $f$ is flat by hypothesis, we deduce that the morphism
\[
f^*\mathcal{O}_Y \longrightarrow \mathcal{O}_X
\]
is injective. Consequently, the canonical map
\[
\Gamma(X, f^*\mathcal{O}_Y) \longrightarrow B = \Gamma(X, \mathcal{O}_X)
\]
is injective. It follows that the canonical map $I \otimes_A B \longrightarrow B$ is injective for all ideals $I$, so $A \longrightarrow B$ is flat.

\begin{cor}[Mario Raimondo, I.3.1]
Let $U$ be an open semi-algebraic subset of $\mathbb{R}^n$, $N(U)$ the ring of NASH functions on $U$, and $\mathcal{O}(U)$ the ring of analytic functions on $U$. Then the canonical homomorphism
\[
N(U) \longrightarrow \mathcal{O}(U)
\]
is faithfully flat.
\end{cor}

\begin{proof}
$N(U)$ is Noetherian, and the morphism
\[
X = (U, \mathcal{O}_U) \xrightarrow{f} Y = \operatorname{Spec}(N(U))
\]
is flat, where $\mathcal{O}_U$ is the sheaf of analytic functions on $U$. Moreover, $H^1(X, f^*F) = 0$ by Cartan’s theorem B. Thus the homomorphism $N(U) \to \mathcal{O}(U)$ is flat and faithfully flat, since every maximal ideal of $N(U)$ is the trace of a maximal ideal of $\mathcal{O}(U)$.
\end{proof}

\begin{cor}[I.3.2]
Let $V$ be a nonsingular algebraic subvariety of $\mathbb{R}^n$ (cf. Chap. II), $N(V)$ the ring of NASH functions on $V$, and $\mathcal{O}(V)$ the ring of analytic functions on $V$. Then the canonical homomorphism
\[
N(V) \longrightarrow \mathcal{O}(V)
\]
is faithfully flat.
\end{cor}

\begin{theorem}[I.4]
Let $(X, \mathcal{O}_X)$ be a locally ringed space, $B = \Gamma(X, \mathcal{O}_X)$, $Y = \operatorname{Spec}(B)$, and $\varphi : X \to Y$ the canonical morphism. Suppose:
\begin{enumerate}
\item $\varphi$ is flat;
\item $B$ is a finite product of integral domains;
\item $\varphi(X)$ contains $\mathrm{Max}(B)$.
\end{enumerate}
Then the functor
\[
G \longmapsto \varphi^* G
\]
from the category of $\mathcal{O}_Y$–modules of finite presentation (resp. quasi–coherent) and torsion–free to the category of $\mathcal{O}_X$–modules that are $A$–coherent and torsion–free is an equivalence of categories.
\end{theorem}

\begin{proof}
We may assume that $A$ is an integral domain. To prove the theorem, it suffices to show that if $F$ is an $\mathcal{O}_X$–module, $A$–coherent and torsion–free, then $F \cong \varphi^* G$ where $G = \Gamma(X, F)^{\sim}$.

Indeed, if $G'$ is a quasi–coherent $\mathcal{O}_Y$–module of finite presentation such that $F \cong \varphi^* G'$, we deduce a morphism
\[
g : G' \longrightarrow G
\]
which is an isomorphism by hypotheses (1) and (3), since $\varphi^* g$ is an isomorphism.

Moreover, for $F' = \Gamma(X, F)^{\sim}$, the morphism
\[
\mathcal{O}_Y^{(I)} \xrightarrow{u} F' \xrightarrow{v} \varphi_* F
\]
is a monomorphism of $\mathcal{O}_Y$–modules of finite presentation, torsion–free, hence split injective, so $u$ is a monomorphism. Therefore $u$ is both a monomorphism and an epimorphism, since $F$ is generated by its global sections. Thus $u$ is an isomorphism of sheaves.

It follows that $F \cong \varphi^* G$ with $G = \Gamma(X, F)^{\sim}$, and the equivalence of categories follows.
\end{proof}

\begin{cor}[I.4.1]
Let $U$ be a semi–algebraic open subset of $\mathbb{R}^n$, $N(U)$ the ring of NASH functions on $U$, and $\mathcal{N}_U$ the sheaf of NASH functions on $U$. Then the functor
\[
F \longmapsto \Gamma(U, F)
\]
is an equivalence between the category of $\mathcal{N}_U$–modules that are $A$–quasi–coherent, torsion–free, and the category of $N(U)$–modules that are quasi–coherent, torsion–free, and of finite type.
\end{cor}

\begin{cor}[I.4.2]
Let $V$ be a nonsingular algebraic variety, $N(V)$ the ring of NASH functions on $V$, and $\mathcal{N}_V$ the sheaf of NASH functions on $V$. Then the functor
\[
F \longmapsto \Gamma(V, F)
\]
is an equivalence between the category of $\mathcal{N}_V$–modules that are $A$–quasi–coherent, torsion–free, and the category of $N(V)$–modules that are quasi–coherent, torsion–free, and of finite type.
\end{cor}

\begin{cor}[I.4.3]
Under the assumptions of (I.4), if $F$ is an $\mathcal{O}_X$–module, $A$–quasi–coherent and flat, then $\Gamma(X, F)$ is a flat $\Gamma(X, \mathcal{O}_X)$–module.
\end{cor}

\begin{proof}
Let $G = \Gamma(X, F)^{\sim}$ be the sheaf on $Y$ associated to the $B$–module $\Gamma(X, F)$. By (I.4) and its proof, $F \cong \varphi^* G$. Since $G$ is a flat $\mathcal{O}_Y$–module by hypotheses (1) and (3), it follows that $\Gamma(Y, G) = \Gamma(X, F)$ is a flat $B$–module.
\end{proof}

\begin{cor}[I.4.4]
Under the assumptions of (I.4.1), if $F$ is an $\mathcal{N}_U$–module, $A$–quasi–coherent and flat, then $\Gamma(U, F)$ is a flat $N(U)$–module.
\end{cor}

\begin{cor}[I.4.5]
Under the assumptions of (I.4.2), if $F$ is an $\mathcal{N}_V$–module, $A$–quasi–coherent and flat, then $\Gamma(V, F)$ is a flat $N(V)$–module.
\end{cor}



\end{document}